\documentclass[11pt,reqno]{amsart}

\scrollmode

\usepackage{times}
\usepackage{amsmath,graphicx}
\usepackage{latexsym}
\usepackage{amscd,amsthm,amssymb,bm}
\usepackage[all]{xy}

\newtheorem{thm}{Theorem}[section]
\newtheorem{prop}[thm]{Proposition}
\newtheorem{lem}[thm]{Lemma}
\newtheorem{cor}[thm]{Corollary}

\theoremstyle{definition}
\newtheorem{ex}[thm]{Example}
\newtheorem{defn}[thm]{Definition}
\newtheorem*{conv}{Conventions}

\theoremstyle{remark}
\newtheorem{rem}[thm]{Remark}

\numberwithin{equation}{section}

\newcommand{\nablaf}{\nabla^{\scriptscriptstyle f}}

\newcommand{\Rf}{R^{\scriptscriptstyle f}}

\newcommand{\Dirf}{D^{\scriptscriptstyle f}}
\newcommand{\Dirfp}{D^{{\scriptscriptstyle f}\prime}}
\newcommand{\Dirfpp}{D^{{\scriptscriptstyle f}\prime\prime}}

\newcommand{\Dirft}{D^{\scriptscriptstyle f_t}}

\newcommand{\Dirbfpp}{\bar{D}^{{\scriptscriptstyle f}\prime\prime}}

\newcommand{\Dolfp}{\bar{\partial}^{{\scriptscriptstyle f}\prime}}
\newcommand{\Dolfpp}{\bar{\partial}^{{\scriptscriptstyle f}\prime\prime}}

\newcommand{\Dolfps}{\bar{\partial}^{{\scriptscriptstyle f}\prime*}}
\newcommand{\Dolfpps}{\bar{\partial}^{{\scriptscriptstyle f}\prime\prime*}}

\newcommand{\Kh}{K^{\scriptscriptstyle\frac{1}{2}}}
\newcommand{\Kbh}{\bar{K}^{\scriptscriptstyle\frac{1}{2}}}
\newcommand{\Kmh}{K^{-\scriptscriptstyle\frac{1}{2}}}
\newcommand{\Kbmh}{\bar{K}^{-\scriptscriptstyle\frac{1}{2}}}

\title{$J$-holomorphic curves and Dirac-harmonic maps}
\author{M.~J.~D.~Hamilton}
\address{      Fachbereich Mathematik\\
               Universit\"at Stuttgart\\
               Pfaffenwaldring 57\\
               70569 Stuttgart\\
               Germany}
\email{mark.hamilton@math.lmu.de}
\date{\today}

\begin{document}

\begin{abstract} Dirac-harmonic maps are critical points of a fermionic action functional, generalizing the Dirichlet energy for harmonic maps. We consider the case where the source manifold is a closed Riemann surface with the canonical $\mathrm{Spin}^c$-structure determined by the complex structure and the target space is a K\"ahler manifold. If the underlying map $f$ is a $J$-holomorphic curve, we determine a space of spinors on the Riemann surface which form Dirac-harmonic maps together with $f$. For suitable complex structures on the target manifold the tangent bundle to the moduli space of $J$-holomorphic curves consists of Dirac-harmonic maps. We also discuss the relation to the A-model of topological string theory.
\end{abstract}

\maketitle

\section{Introduction}
We briefly recall the definition of Dirac-harmonic maps (see \cite{CJLW1, CJLW2} and Section \ref{sect:background} for more details). Let $(\Sigma,h)$ and $(M,g)$ be Riemannian manifolds, where $\Sigma$ is closed and oriented, and $f\colon \Sigma\rightarrow M$ a smooth map. We assume that $\Sigma$ is a spin manifold and choose a spin structure $\mathfrak{s}$ with associated complex spinor bundle $S$. We can then form the twisted spinor bundle $S\otimes_{\mathbb{R}} f^*TM$ of spinors on $\Sigma$ with values in the pullback $f^*TM$ (also called spinors along the map $f$). The Dirac operator
\begin{equation*}
\Dirf\colon \Gamma(S\otimes_{\mathbb{R}} f^*TM)\longrightarrow \Gamma(S\otimes_{\mathbb{R}} f^*TM)
\end{equation*} 
is determined by the Levi--Civita connections on $\Sigma$ and $M$. 

Dirac-harmonic maps $(f,\psi)$, where $\psi\in\Gamma(S\otimes_{\mathbb{R}} f^*TM)$, are solutions of the following system of coupled equations \cite{CJLW1, CJLW2}: 
\begin{equation}\label{eqn:Dirac-harmonic map intro}
\begin{split}
\tau(f)&=\mathcal{R}(f,\psi)\\
\Dirf\psi&=0.
\end{split}
\end{equation}
Here $\tau(f)\in\Gamma(f^*TM)$ is the so-called tension field of $f$ (cf.~\cite{ES} and equation \eqref{eqn:tension field tau f}). The curvature term $\mathcal{R}(f,\psi)$ is determined by the curvature tensor $R$ of the Riemannian metric $g$ on $M$ and is an algebraic expression in the differential $df$ and the spinor $\psi$ (linear in $df$ and quadratic in $\psi$); see Appendix \ref{sect:appendix B} for a definition. 

The system of equations \eqref{eqn:Dirac-harmonic map intro} for Dirac-harmonic maps makes sense more generally if we replace the spin structure $\mathfrak{s}$ by a $\mathrm{Spin}^c$-structure $\mathfrak{s}^c$ and consider twisted spinors $\psi\in\Gamma(S^c\otimes_{\mathbb{R}} f^*TM)$, where $S^c$ is the complex spinor bundle associated to $\mathfrak{s}^c$. For the definition of the Dirac operator
\begin{equation*}
\Dirf\colon \Gamma(S^c\otimes_{\mathbb{R}} f^*TM)\longrightarrow \Gamma(S^c\otimes_{\mathbb{R}} f^*TM)
\end{equation*}
one has to choose (in addition to the Riemannian metrics $h$ and $g$) a Hermitian connection on the characteristic complex line bundle of $\mathfrak{s}^c$. We assume throughout that such a choice has been made and fixed (in the case we discuss there is a canonical choice of such a connection, determined by the Riemannian metric $h$).

Every almost Hermitian manifold $(\Sigma,j,h)$ has a canonical $\mathrm{Spin}^c$-structure $\mathfrak{s}^c$ whose associated spinor bundle $S^c=S^{c+}\oplus S^{c-}$ is the direct sum of the positive and negative Weyl spinor bundles
\begin{align*}
S^{c+}&=\Lambda^{0,\mathrm{even}}\\
S^{c-}&=\Lambda^{0,\mathrm{odd}}.
\end{align*} 
We focus on the special case where $(\Sigma,j,h)$ is a closed Riemann surface, so that
\begin{align*}
S^{c+}&=\Lambda^{0,0}=\underline{\mathbb{C}}\\
S^{c-}&=\Lambda^{0,1}=K^{-1}.
\end{align*} 
The Levi--Civita connection $\nabla^h$ induces a connection on $S^c$ with Dirac operator 
\begin{equation*}
D\colon \Gamma(S^{c\pm})\longrightarrow \Gamma(S^{c\mp})
\end{equation*}
equal to the classical Dolbeault--Dirac operator
\begin{equation*}
\sqrt{2}(\bar{\partial}+\bar{\partial}^*).
\end{equation*}
Suppose that the target space $(M^{2n},J,g,\omega)$ is an almost Hermitian manifold with a Hermitian connection $\nabla^M$ and $f\colon\Sigma\rightarrow M$ a smooth map. 
We first derive a formula for the twisted Dirac operator
\begin{equation*}
\Dirf\colon \Gamma(S^{c\pm}\otimes_{\mathbb{R}} f^*TM)\longrightarrow \Gamma(S^{c\mp}\otimes_{\mathbb{R}} f^*TM).
\end{equation*}
\begin{prop}\label{prop:twisted Dirac Df formula}
The spinor bundle $S^c\otimes_{\mathbb{R}} f^*TM$ decomposes into two twisted complex spinor bundles
\begin{equation*}
S^c\otimes_{\mathbb{R}} f^*TM=(S^c\otimes_{\mathbb{C}} f^*T^{1,0}M)\oplus (S^c\otimes_{\mathbb{C}} f^*T^{0,1}M).
\end{equation*}
There is a corresponding decomposition $\Dirf=\Dirfp+\Dirfpp$ of the Dirac operator into two twisted Dolbeault--Dirac operators
\begin{align*}
\Dirfp&=\sqrt{2}(\Dolfp+\Dolfps)\\
\Dirfpp&=\sqrt{2}(\Dolfpp+\Dolfpps).
\end{align*}
The Hirzebruch--Riemann--Roch Theorem implies for the indices
\begin{align*}
\mathrm{ind}_{\mathbb{C}}\Dirfp&=n(1-g_\Sigma)+c_1(A)\\
\mathrm{ind}_{\mathbb{C}}\Dirfpp&=n(1-g_\Sigma)-c_1(A),
\end{align*}
where $g_\Sigma$ is the genus of $\Sigma$, $A=f_*[\Sigma]\in H_2(M;\mathbb{Z})$ is the integral homology class represented by $\Sigma$ under $f$ and $c_1(A)=\langle c_1(TM,J),A\rangle$.
\end{prop}
We then restrict to the case where $(M,J,g,\omega)$ is K\"ahler, $\nabla^M=\nabla^g$ the Levi--Civita connection and $f\colon\Sigma\rightarrow M$ a $J$-holomorphic curve. In this case, the Dolbeault operator $\Dolfp$ is equal to the linearization $L_f\bar{\partial}_J$ of the non-linear Cauchy--Riemann operator $\bar{\partial}_J$ in $f$. In particular, the kernel of $\Dirfp$ is given by the direct sum of the deformation and obstruction space for the $J$-holomorphic curve $f$ (cf.~
Remark \ref{rem:discussion J integrable regular}):
\begin{align*}
\mathrm{ker}\,\Dirfp&= \mathrm{ker}\,\Dolfp\oplus \mathrm{ker}\,\Dolfps\\
&\cong \mathrm{Def}_J(f)\oplus \mathrm{Obs}_J(f).
\end{align*}
The pair $(f,J)$ is called regular if $\mathrm{Obs}_J(f)=0$.
\begin{thm}\label{thm:J-hol curve Dirac-harmonic}
Suppose that $(M,J,g,\omega)$ is a K\"ahler manifold of complex dimension $n>0$ and $f\colon \Sigma\rightarrow M$ a $J$-holomorphic curve. If $\psi\in \Gamma(S^c\otimes_{\mathbb{C}} f^*T^{\mathbb{C}}M)$ is an element of one of the following vector spaces, then $(f,\psi)$ is Dirac-harmonic:
\begin{equation}\label{eqn:J-hol curve Dirac-harmonic vector spaces}
\begin{array}{ll} \mathrm{ker}\,\Dolfp\oplus \mathrm{ker}\,\Dolfpp, & \mathrm{ker}\,\Dolfps\oplus \mathrm{ker}\,\Dolfpps \\ \mathrm{ker}\,\Dolfp\oplus \mathrm{ker}\,\Dolfpps, & \mathrm{ker}\,\Dolfpp\oplus \mathrm{ker}\,\Dolfps.  \end{array}
\end{equation}
At least one of these vector spaces is non-zero, except possibly in the case that $g_\Sigma=1$ and $c_1(A)=0$. 
\end{thm} 
\begin{cor}\label{cor:moduli M(A,J) Dirac-harmonic}
Let $(\Sigma,j)$ be a Riemann surface, $A\in H_2(M;\mathbb{Z})$ and denote by $\mathcal{M}(A,J)$ the moduli space of all $J$-holomorphic curves $f\colon\Sigma\rightarrow M$ with $f_*[\Sigma]=A$. Suppose that $(f,J)$ is regular for all $f\in\mathcal{M}(A,J)$. Then $\mathcal{M}(A,J)$ is a smooth manifold (possibly empty) of dimension
\begin{equation*}
\dim_{\mathbb{R}}\mathcal{M}(A,J)=2n(1-g_\Sigma)+2c_1(A).
\end{equation*}
Every element $(f,\psi)\in T\mathcal{M}(A,J)$ of the tangent bundle of the moduli space is a Dirac-harmonic map.
\end{cor}
\begin{rem}
For the case of spin structures the vector spaces corresponding to the ones in \eqref{eqn:J-hol curve Dirac-harmonic vector spaces} appear in the proof of \cite[Theorem 1.1]{Sun}.
\end{rem}
\begin{rem}
Dirac-harmonic maps for the canonical $\mathrm{Spin}^c$-structure on Riemann surfaces are closely related to the A-model of topological string theory \cite{W1, W2} (with a fixed metric $h$, i.e.~without worldsheet gravity); see Section \ref{sect:TQFT} for a short discussion. In particular, in the A-model path integrals of certain operators localize to integrals over the finite-dimensional moduli spaces $\mathcal{M}(A,J)$ and the tangent bundle $T\mathcal{M}(A,J)$ can be identified with the space of $\chi$-zero modes (in our notation $\chi=\psi\in \mathrm{ker}\,\Dolfp$).
\end{rem}
In the last section we consider a generalization of Theorem \ref{thm:J-hol curve Dirac-harmonic} to twisted $\mathrm{Spin}^c$-structures $S^c\otimes_{\mathbb{C}} L$ with a holomorphic line bundle $L\rightarrow \Sigma$; see Corollary \ref{cor:Dirac-harmonic twisted hol L}. For $L=K^{\scriptscriptstyle\frac{1}{2}}$ this includes the case of the spinor bundle $S=S^c\otimes_{\mathbb{C}} K^{\scriptscriptstyle\frac{1}{2}}$ of a spin structure $\mathfrak{s}$.

Dirac-harmonic maps $(f,\psi)$ from surfaces $\Sigma$ with a spin structure to Riemannian target manifolds $M$ have been studied before. We summarize some of the results in \cite{AG1, AG2, CJLW2, CJSZ, JMZ, Mo, Sun, Yang}. 

Examples of Dirac-harmonic maps for $\Sigma=M=S^2$ were constructed in \cite{CJLW2} where $f$ is a conformal map and $\psi$ is defined using a twistor spinor on $S^2$. This method was generalized in \cite{JMZ} to arbitrary Riemann surfaces $\Sigma$ admitting twistor spinors and arbitrary Riemannian manifolds $M$, where the map $f$ is harmonic (among closed surfaces only $S^2$ and $T^2$ admit non-zero twistor spinors \cite[A.2.2]{Gin}). In \cite{Yang} and \cite{CJSZ} it was shown that all Dirac-harmonic maps with source $\Sigma$ of genus $g_\Sigma$ and target $M=S^2$, so that $|\deg(f)|+1>g_\Sigma$, can be obtained using the constructions from \cite{CJLW2, JMZ}, where $f$ is holomorphic or antiholomorphic and $\psi$ is defined using a twistor spinor on $\Sigma$, possibly with isolated singularities (see also \cite{Mo}).

Dirac-harmonic maps $(f,\psi)$ from spin K\"ahler manifolds to arbitrary K\"ahler manifolds were studied in \cite{Sun}. In Example \ref{ex:Dirac harmonic spin structure} below we consider the case where the source is a Riemann surface $\Sigma$ with a spin structure and the map $f$ is $J$-holomorphic.

Existence results for Dirac-harmonic maps related to the $\alpha$-genus $\alpha(\Sigma,\mathfrak{s},f)$ for a spin structure $\mathfrak{s}$ on $\Sigma$ were discussed in \cite{AG1}. Section 10.1 in \cite{AG1} contains several results for Dirac-harmonic maps from surfaces to Riemannian manifolds $M$ of dimension $\geq 3$. In \cite{AG2} Dirac-harmonic maps from surfaces to Riemannian manifolds were constructed with methods related to an ansatz in \cite{JMZ}.

In \cite{G1, G2, G3} another fermionic generalization of $J$-holomorphic curves was studied (see Remark \ref{rem:hol supercurve} for a brief discussion of the relation to Dirac-harmonic maps).

\begin{conv} In the following, all Riemann surfaces $\Sigma$ are closed (compact and without boundary), connected and oriented by the complex structure. For Riemannian metrics $h$ on $\Sigma$ and $g$ on $M$ we denote by $\nabla^h$ and $\nabla^g$ the Levi--Civita connections. Tensor products of vector spaces and vector bundles are over the complex numbers $\mathbb{C}$, unless indicated otherwise.
\end{conv}

\section{Some background on Dirac-harmonic maps}\label{sect:background}
Recall that harmonic maps $f\colon \Sigma\rightarrow M$ from a closed, oriented Riemannian manifold $(\Sigma,h)$ to a Riemannian manifold $(M,g)$ are smooth maps, defined as the critical points of the Dirichlet energy functional \cite{ES}
\begin{equation}\label{eqn:L[f]}
L[f]=\frac{1}{2}\int_\Sigma|df|^2\,\mathrm{dvol}_h,
\end{equation}
where $df$ is the differential of $f$ and $|df|^2$ is the length-squared determined by the metrics $h$ and $g$. The Euler--Lagrange equation for stationary points of $L[f]$ under variations of $f$ is
\begin{equation*}
\tau(f)=0,
\end{equation*}
where $\tau(f)$ is the tension field
\begin{equation}\label{eqn:tension field tau f}
\tau(f)=\mathrm{tr}_h(\nablaf df)=\sum_{\alpha}(\nablaf_{e_\alpha}df)(e_\alpha).
\end{equation}
Here $df$ is considered as an element of $\Omega^1(f^*TM)$ and the connection $\nablaf$ on the vector bundle $f^*TM\rightarrow\Sigma$ is induced from the Levi--Civita connection $\nabla^M=\nabla^g$. The basis $\{e_\alpha\}$ is a local orthonormal frame on $\Sigma$.
\begin{rem}
If the connection $\nabla^M$ on $M$ is compatible with $g$, but not torsion-free, then harmonic maps $f$ do not necessarily satisfy $\tau(f)=0$. 
\end{rem}
Suppose that $\Sigma$ is a spin manifold and let $\mathfrak{s}$ be a spin structure on $\Sigma$ with associated complex spinor bundle $S$ and twisted spinor bundle $S\otimes_{\mathbb{R}} f^*TM$. Note that if $V$ is a complex vector space and $W$ a real vector space, then $V\otimes_{\mathbb{R}}W$ is a complex vector space isomorphic to $V\otimes_{\mathbb{C}}W^{\mathbb{C}}$, where $W^{\mathbb{C}}$ is the complexification $W\otimes_{\mathbb{R}}\mathbb{C}$. It follows that there is a (canonical) isomorphism of complex vector bundles
\begin{equation*}
S\otimes_{\mathbb{R}} f^*TM\cong S\otimes_{\mathbb{C}}f^*T^{\mathbb{C}}M,
\end{equation*}
with $T^{\mathbb{C}}M=TM\otimes_{\mathbb{R}}\mathbb{C}$ (see \cite[Section 2]{Yang}).

The Levi--Civita connection on $\Sigma$ and the connection $\nablaf$ on $f^*TM$ yield a Dirac operator
\begin{align*}
\Dirf\colon \Gamma(S\otimes_{\mathbb{R}} f^*TM)\longrightarrow \Gamma(S\otimes_{\mathbb{R}} f^*TM).
\end{align*} 
Dirac-harmonic maps $(f,\psi)$ are defined as the critical points of the fermionic action functional \cite{CJLW1, CJLW2}
\begin{equation}\label{eqn:L[f,psi]}
L[f,\psi]=\frac{1}{2}\int_\Sigma\left(|df|^2+\langle \psi, \Dirf\psi\rangle\right)\,\mathrm{dvol}_h.
\end{equation}
A pair $(f,\psi)$ is Dirac-harmonic if and only if it is a solution of the system of coupled Euler--Lagrange equations \eqref{eqn:Dirac-harmonic map intro} (see \cite[Proposition 2.1]{CJLW2} for a proof of the formulae below):
\begin{itemize}
\item If $f$ is fixed and $\psi_t$ a variation of $\psi$ with
\begin{equation*}
\psi_0=\psi,\quad \left.\frac{d\psi_t}{dt}\right|_{t=0}=\eta\in \Gamma(S^c\otimes_{\mathbb{R}} f^*TM),
\end{equation*}
then
\begin{equation*}
\left.\frac{d}{dt}\right|_{t=0}\int_\Sigma \langle \psi_t, \Dirf\psi_t\rangle\,\mathrm{dvol}_h=2\int_\Sigma \langle\eta,\Dirf \psi\rangle\,\mathrm{dvol}_h.
\end{equation*}
\item If $f_t$ is a variation of $f$ with
\begin{equation*}
f_0=f,\quad \left.\frac{df_t}{dt}\right|_{t=0}=f^*X\in \Gamma(f^*TM),
\end{equation*}
then
\begin{equation*}
\left.\frac{d}{dt}\right|_{t=0}\int_\Sigma |df_t|^2\,\mathrm{dvol}_h=-2\int_\Sigma g(\tau(f),f^*X)\,\mathrm{dvol}_h.
\end{equation*}
Suppose in addition that $\psi_t=\sum_\mu \psi_\mu\otimes f_t^*\partial_\mu$ is a twisted spinor with time-independent components $\psi_\mu$ with respect to local coordinates $\{x_\mu\}$ (or a local frame) of $M$. If $\psi=\psi_0$ satisfies $\Dirf\psi=0$, then
\begin{equation}\label{eqn:var psi,Dpsi R(f,psi)}
\left.\frac{d}{dt}\right|_{t=0}\int_\Sigma \langle \psi_t, \Dirft\psi_t\rangle\,\mathrm{dvol}_h=2\int_\Sigma g(\mathcal{R}(f,\psi),f^*X)\,\mathrm{dvol}_h.
\end{equation}
More details on the calculation of this variation can be found in Appendix \ref{sect:appendix B}.
\end{itemize}
Dirac-harmonic maps are generalizations of harmonic maps: For the trivial spinor $\psi\equiv 0$, the curvature term $\mathcal{R}(f,\psi)$ vanishes identically and the system of equations \eqref{eqn:Dirac-harmonic map intro} reduces to the equation
\begin{equation*}
\tau(f)=0,
\end{equation*}
i.e.~$(f,0)$ is Dirac-harmonic for any harmonic map $f$.

The fermionic action functional \eqref{eqn:L[f,psi]} is motivated by theoretical physics: Suppose that $\Sigma$ is $2$-dimensional and $h,g$ Lorentzian metrics. The Dirichlet energy $L[X]$ for smooth maps $X\colon\Sigma\rightarrow M$ is (up to a normalization constant) the non-linear $\sigma$-model (Polyakov) action for bosonic strings propagating in $(M,g)$, cf.~\cite{CMPF}. 

The functional $L[X,\psi]$ for Dirac-harmonic maps is part of the supersymmetric non-linear $\sigma$-model action \cite{AF}: Choosing coordinates $\{x_\mu\}$ on an open subset $U\subset M$ we can write every spinor $\psi\in \Gamma(S\otimes_{\mathbb{R}} f^*TM)$ on $\tilde{U}=f^{-1}(U)$ as
\begin{equation*}
\psi=\sum_{\mu} \psi_\mu\otimes f^*\partial_\mu,\quad\text{with}\quad\psi_\mu\in \Gamma(\tilde{U},S).
\end{equation*}
The spinors $\psi_\mu$ are the fermionic superpartners of the scalar fields $X_\mu\in\mathcal{C}^{\infty}(\tilde{U},\mathbb{R})$, i.e.~the coordinate fields of the map $X$ (in physics, the spinors $\psi_\mu$ take values in a Grassmann algebra). 

In the supersymmetric non-linear $\sigma$-model action in \cite{AF} there is an additional curvature term which is determined by the curvature tensor $R$ of $g$ and of order $4$ in the spinor $\psi$ (cf.~\cite{CJW}). The full action for superstrings contains also a gravitino $\chi$, the superpartner of the metric $h$. This action was studied from a mathematical point of view in \cite{JKTWZ}.

\section{$\mathrm{Spin}^c$-structures on Riemann surfaces}
We discuss some background material concerning $\mathrm{Spin}^c$-structures on Riemann surfaces (more details can be found e.g.~in \cite{H, BGV, Fr, LM}).

Let $(\Sigma,j,h)$ be a closed Riemann surface with complex structure $j$ and compatible Riemannian metric $h$. The canonical $\mathrm{Spin}^c$-structure $\mathfrak{s}^c$ on $\Sigma$ has spinor bundles
\begin{align*}
S^{c+}&=\Lambda^{0,0}=\underline{\mathbb{C}}\\
S^{c-}&=\Lambda^{0,1}=K^{-1},
\end{align*} 
where $\underline{\mathbb{C}}$ is the trivial complex line bundle and $K^{-1}=\bar{K}$ is the anticanonical line bundle. The spaces of smooth sections are
\begin{align*}
\Gamma(S^{c+})&=\mathcal{C}^{\infty}(\Sigma,\mathbb{C})\\
\Gamma(S^{c-})&=\Omega^{0,1}(\Sigma).
\end{align*} 
Our notation for tangent vectors and $1$-forms of type $(1,0)$ and $(0,1)$ can be found in Appendix \ref{sect:appendix A}. The Riemannian metric $h$ extends to Hermitian bundle metrics on $T^{1,0}\oplus T^{0,1}$ and $\Lambda^{1,0}\oplus\Lambda^{0,1}$ and the choice of a local $h$-orthonormal basis $(e_1,e_2)$ of $T\Sigma$  with $e_2=je_1$ determines local unit basis vectors
\begin{equation*}
\epsilon\in T^{1,0},\quad \bar{\epsilon} \in T^{0,1}
\end{equation*}
and dual unit basis $1$-forms
\begin{equation*}
\kappa\in \Lambda^{1,0},\quad \bar{\kappa} \in \Lambda^{0,1}.
\end{equation*}
Any element $\beta\in\Lambda^{0,1}$ can be written as
\begin{equation}\label{eqn:beta 0,1 expansion}
\beta=\sqrt{2}\beta(e_1)\bar{\kappa}.
\end{equation}
The spinor bundle $S^c$ has a Clifford multiplication
\begin{equation*}
\gamma\colon T\Sigma\times S^{c\pm}\longrightarrow S^{c\mp},\quad (v,\psi)\longmapsto \gamma(v)\psi=v\cdot \psi,
\end{equation*}
that satisfies the Clifford relation
\begin{equation*}
v\cdot w\cdot \psi + w\cdot v\cdot \psi = -2h(v,w)\psi.
\end{equation*}
Let $\alpha\in (T^{\mathbb{C}}\Sigma)^*$. For $\phi\in \underline{\mathbb{C}}=\Lambda^{0,0}$ Clifford multiplication is given by
\begin{equation*}
\alpha\cdot\phi=\sqrt{2}\alpha^{0,1}\phi,
\end{equation*}
which implies
\begin{align*}
e_1\cdot \phi &= \phi\bar{\kappa}\\
e_2\cdot \phi &= i\phi\bar{\kappa}.
\end{align*}
For $\beta\in K^{-1}=\Lambda^{0,1}$ Clifford multiplication is given by contraction
\begin{equation*}
\alpha\cdot \beta = -\sqrt{2}i_{\overline{\alpha^{1,0}}}\beta,
\end{equation*}
implying
\begin{align*}
e_1\cdot \beta &= -\beta(\bar{\epsilon})\\
e_2\cdot \beta &= i\beta(\bar{\epsilon}).
\end{align*}
In particular, the volume form $\mathrm{dvol}_h=e_1^*\wedge e_2^*$ acts as
\begin{equation}\label{eqn:Clifford volume}
\mathrm{dvol}_h=\pm (-i)\quad\text{on}\quad S^{c\pm}.
\end{equation}
The decomposition of the differential
\begin{equation*}
d\colon \mathcal{C}^{\infty}(\Sigma,\mathbb{C})\longrightarrow \Omega^1(\Sigma,\mathbb{C})
\end{equation*}
into $(1,0)$- and $(0,1)$-components is denoted by
\begin{equation*}
d\phi=(d\phi)^{1,0}+(d\phi)^{0,1}=\partial\phi+\bar{\partial}\phi
\end{equation*}
and the Dolbeault operator is given by
\begin{equation*}
\bar{\partial}\colon \mathcal{C}^{\infty}(\Sigma,\mathbb{C})\longrightarrow \Omega^{0,1}(\Sigma),\quad \bar{\partial}\phi=\tfrac{1}{2}(d\phi+id\phi\circ j)
\end{equation*}
with formal adjoint
\begin{equation*}
\bar{\partial}^*\colon \Omega^{0,1}(\Sigma)\longrightarrow\mathcal{C}^{\infty}(\Sigma,\mathbb{C}).
\end{equation*}
The Levi--Civita connection $\nabla^h$ of the K\"ahler metric $h$ satisfies $\nabla^hj=j\nabla^h$ and induces a connection on $K^{-1}$ and thus a Hermitian connection on $S^c$, compatible with Clifford multiplication. We consider the associated Dirac operator 
\begin{equation*}
D\colon \Gamma(S^{c\pm})\longrightarrow \Gamma(S^{c\mp}).
\end{equation*}
\begin{lem}[cf.~\cite{H}]
The Dirac operator $D$ is equal to the Dolbeault--Dirac operator
\begin{equation*}
\sqrt{2}(\bar{\partial}+\bar{\partial}^*).
\end{equation*}
The Riemann--Roch theorem implies for the index
\begin{equation*}
\mathrm{ind}_{\mathbb{C}}D=1-g_\Sigma,
\end{equation*}
where $g_\Sigma$ is the genus of $\Sigma$.
\end{lem}
\begin{proof}
Let $\phi\in \mathcal{C}^{\infty}(\Sigma,\mathbb{C})$ be a positive spinor. On $\mathcal{C}^{\infty}(\Sigma,\mathbb{C})$ the connection is just the differential $d$, hence
\begin{align*}
D\phi&=e_1\cdot d\phi(e_1)+e_2\cdot d\phi(e_2)=(d\phi(e_1)+id\phi(e_2))\bar{\kappa}\\
&=2\bar{\partial}\phi(e_1)\bar{\kappa}=\sqrt{2}\bar{\partial}\phi,
\end{align*}
where the last step follows from equation \eqref{eqn:beta 0,1 expansion}. Thus 
\begin{align*}
D\colon \mathcal{C}^{\infty}(\Sigma,\mathbb{C})&\longrightarrow \Omega^{0,1}(\Sigma) \\
\phi&\longmapsto \sqrt{2}\bar{\partial}\phi. 
\end{align*}
Since the Dirac operator is formally self-adjoint, the claim follows.
\end{proof}
\begin{rem}\label{rem:Riemann surface spin str}
Riemann surfaces are spin, hence we can choose a spin structure $\mathfrak{s}$ on $\Sigma$, which is equivalent to the choice of a holomorphic square root $K^{\scriptscriptstyle\frac{1}{2}}$ of the canonical bundle $K$ (see \cite{At, H}). The spinor bundles of $\mathfrak{s}$ are
\begin{align*}
S^{+}&=K^{\scriptscriptstyle\frac{1}{2}}\\
S^{-}&=K^{-\scriptscriptstyle\frac{1}{2}}
\end{align*} 
and the spinor bundle of the canonical $\mathrm{Spin}^c$-structure is obtained by twisting
\begin{equation*}
S^c=S\otimes K^{-\scriptscriptstyle\frac{1}{2}}.
\end{equation*}
There is another $\mathrm{Spin}^c$-structure with spinor bundle
\begin{equation*}
\bar{S}^c=S\otimes K^{\scriptscriptstyle\frac{1}{2}},
\end{equation*}
i.e. 
\begin{align*}
\bar{S}^{c+}&=K\\
\bar{S}^{c-}&=\underline{\mathbb{C}}.
\end{align*}
\end{rem}
\begin{rem}\label{rem:twist Sc with L}
Let $L\rightarrow\Sigma$ be a complex line bundle with a Hermitian bundle metric. Then there is a twisted $\mathrm{Spin}^c$-structure $\mathfrak{s}^c\otimes L$ with spinor bundles
\begin{align*}
S^{c+}\otimes L&=L\\
S^{c-}\otimes L&=K^{-1}\otimes L.
\end{align*}
A connection $\nabla^B$ on $L$, compatible with the Hermitian bundle metric, together with the Levi--Civita connection $\nabla^h$ yields a Hermitian connection on $S^c\otimes L$ and a Dirac operator
\begin{equation*}
D_B\colon \Gamma(S^{c\pm}\otimes L)\longrightarrow \Gamma(S^{c\mp}\otimes L).
\end{equation*}
With the Dolbeault operator
\begin{equation*}
\bar{\partial}_B\colon \Gamma(L)\longrightarrow \Omega^{0,1}(L),\quad \bar{\partial}_B\phi=\tfrac{1}{2}(\nabla^B\phi+i\nabla^B\phi\circ j)
\end{equation*}
the Dirac operator $D_B$ is equal to the Dolbeault--Dirac operator
\begin{equation*}
\sqrt{2}(\bar{\partial}_B+\bar{\partial}_B^*).
\end{equation*}
\end{rem}

\section{Dirac operator along maps and $J$-holomorphic curves}
Let $(\Sigma,j,h)$ be a Riemann surface and $(M,J,g,\omega)$ an almost Hermitian manifold of real dimension $2n$ with almost complex structure $J$, Riemannian metric $g$ and non-degenerate $2$-form $\omega$, related by
\begin{align*}
g(Jx,Jy)&=g(x,y)\\
\omega(x,y)&=g(Jx,y)\quad\forall x,y\in TM.
\end{align*}
We fix a Hermitian connection $\nabla^M$ on $TM$, i.e.~an affine connection such that $\nabla^M g=0$ and $\nabla^M J=0$. For a general almost Hermitian manifold the connection $\nabla^M$ has non-zero torsion. The Hermitian connection $\nabla^M$ can be chosen torsion-free, hence equal to the Levi--Civita connection $\nabla^g$ of $g$, if and only if $(M,J,g,\omega)$ is K\"ahler.

Let $f\colon \Sigma \rightarrow M$ be a smooth map and consider the pullback $f^*TM\rightarrow \Sigma$ of the tangent bundle $TM$. If $X$ is vector field on $M$, then the pullback
\begin{equation*}
f^*X\colon \Sigma \longrightarrow f^*TM,\quad z\longmapsto X_{f(z)}
\end{equation*}
is a section of $f^*TM$. There is a unique Hermitian connection $\nablaf$ on $f^*TM$ so that
\begin{equation*}
\nablaf_V(f^*X)=f^*(\nabla^M_{df(V)}X)\quad\forall X\in \mathfrak{X}(M), V\in T\Sigma.
\end{equation*} 
We consider the twisted spinor bundle 
\begin{equation*}
S^c\otimes_{\mathbb{R}} f^*TM\cong S^c\otimes f^*T^{\mathbb{C}}M
\end{equation*}
on $\Sigma$. The Riemannian metric $g$ extends to a Hermitian bundle metric $\langle\cdot\,,\cdot\rangle$ on $T^{\mathbb{C}}M$. There is a decomposition into orthogonal $\pm i$-eigenspaces of the complex linear extension of $J$,
\begin{equation*}
T^{\mathbb{C}}M =T^{1,0}M\oplus T^{0,1}M
\end{equation*}
and a corresponding decomposition of $S^c\otimes_{\mathbb{R}} f^*TM$ into two twisted complex spinor bundles (cf.~\cite[Section 3]{Yang})
\begin{equation}\label{eqn:two spinor subbundles Sc fTM}
S^c\otimes_{\mathbb{R}} f^*TM=(S^c \otimes f^*T^{1,0}M)\oplus (S^c\otimes f^*T^{0,1}M)
\end{equation}
(the tensor products on the right are over $\mathbb{C}$). The connection $\nabla^M$ extends to a Hermitian connection on $T^{\mathbb{C}}M$ which preserves both complex subbundles $T^{1,0}M$ and $T^{0,1}M$. The connections $\nabla^h$ and $\nablaf$ thus define a Hermitian connection on $S^c\otimes_{\mathbb{R}} f^*TM$, also denoted by $\nablaf$, which preserves both complex spinor bundles on the right hand side of equation \eqref{eqn:two spinor subbundles Sc fTM}.
\begin{defn}[cf.~\cite{CJLW2}]
The associated twisted Dirac operator
\begin{align*}
\Dirf\colon \Gamma(S^{c\pm}\otimes_{\mathbb{R}} f^*TM)&\longrightarrow \Gamma(S^{c\mp}\otimes_{\mathbb{R}} f^*TM)\\
\psi&\longmapsto \sum_{\alpha=1}^2e_\alpha\cdot \nabla_{e_\alpha}^f\psi
\end{align*}
is called the {\em Dirac operator along the map $f$}. Under the splitting in equation \eqref{eqn:two spinor subbundles Sc fTM} the Dirac operator $\Dirf$ decomposes into two twisted Dirac operators
\begin{align*}
\Dirfp\colon \Gamma(S^{c\pm}\otimes f^*T^{1,0}M)&\longrightarrow \Gamma(S^{c\mp}\otimes f^*T^{1,0}M)\\
\Dirfpp\colon \Gamma(S^{c\pm}\otimes f^*T^{0,1}M)&\longrightarrow \Gamma(S^{c\mp}\otimes f^*T^{0,1}M).
\end{align*}
\end{defn}
Since the connection $\nablaf$ on the twisted spinor bundle is obtained from the Levi--Civita connection $\nabla^h$ on $\Sigma$, the Dirac operator $\Dirf$ is formally self-adjoint. We consider the Dolbeault operators for the complex vector bundles $f^*T^{1,0}M$ and $f^*T^{0,1}M$,
\begin{align*}
\Dolfp\colon \Gamma(f^*T^{1,0}M)&\longrightarrow \Omega^{0,1}(f^*T^{1,0}M)\\
\Dolfpp\colon \Gamma(f^*T^{0,1}M)&\longrightarrow \Omega^{0,1}(f^*T^{0,1}M)
\end{align*}
defined by
\begin{align*}
\Dolfp\psi&=\tfrac{1}{2}(\nablaf\psi+J\circ\nablaf\psi\circ j)\\
\Dolfpp\psi&=\tfrac{1}{2}(\nablaf\psi-J\circ\nablaf\psi\circ j).
\end{align*}
The formal adjoints are denoted by $\Dolfps$ and $\Dolfpps$.
\begin{proof}[Proof of Proposition \ref{prop:twisted Dirac Df formula}]
Let $\psi\in \Gamma(f^*T^{1,0}M)$. Then
\begin{align*}
\Dirfp\psi&=e_1\cdot \nabla_{e_1}^f\psi+e_2\cdot \nabla_{e_2}^f\psi=\bar{\kappa}\otimes (\nabla_{e_1}^f\psi+i\nabla_{e_2}^f\psi)=\bar{\kappa}\otimes (\nabla_{e_1}^f\psi+J\nabla_{je_1}^f\psi)\\
&=\sqrt{2}\Dolfp\psi.
\end{align*}
This implies the claim for the Dirac operator $\Dirfp$, because it is self-adjoint. The claim for $\Dirfpp$ follows similarly.
\end{proof}
Recall that a $J$-holomorphic curve is a smooth map $f\colon \Sigma\rightarrow M$ such that
\begin{equation*}
df\circ j = J\circ df,
\end{equation*}
where 
\begin{equation*}
df\colon T\Sigma\longrightarrow TM
\end{equation*}
is the differential. With the non-linear Cauchy--Riemann operator
\begin{equation*}
\bar{\partial}_Jf=\tfrac{1}{2}(df+J\circ df\circ j),
\end{equation*}
the map $f$ is a $J$-holomorphic curve if and only if 
\begin{equation*}
\bar{\partial}_Jf=0.
\end{equation*}
\begin{cor}\label{cor:J integrable ker Df}
Suppose that $(M,J,g,\omega)$ is K\"ahler, $\nabla^M=\nabla^g$ the Levi--Civita connection and $f\colon \Sigma\rightarrow M$ a $J$-holomorphic curve.
\begin{enumerate} 
\item $T^{1,0}M\cong (TM,J)$ and $f^*T^{1,0}M$ is a holomorphic vector bundle over $\Sigma$.
\item\label{item2:J int} $\Dolfp$ is equal to the linearization $L_f\bar{\partial}_J$ of the non-linear Cauchy--Riemann operator $\bar{\partial}_J$ in $f$.
\item\label{item3:J int} The kernel of $\Dirfp$ is given by
\begin{align*}
\mathrm{ker}\,\Dirfp&= \mathrm{ker}\,\Dolfp\oplus \mathrm{ker}\,\Dolfps\cong \mathrm{ker}\,L_f\bar{\partial}_J\oplus \mathrm{coker}\,L_f\bar{\partial}_J\\
&\cong H^0(\Sigma, f^*T^{1,0}M)\oplus H^1(\Sigma, f^*T^{1,0}M).
\end{align*}
\item\label{item4:J int} The kernel of $\Dirfpp$ is given by
\begin{align*}
\mathrm{ker}\,\Dirfpp&= \mathrm{ker}\,\Dolfpp\oplus \mathrm{ker}\,\Dolfpps\\
&\cong H^1(\Sigma, K_\Sigma\otimes f^*T^{1,0}M)^*\oplus H^0(\Sigma, K_\Sigma\otimes f^*T^{1,0}M)^*.
\end{align*}
\end{enumerate}
\end{cor}
\begin{proof}
The claim in (\ref{item2:J int}) follows from \cite[p.~28]{McS1}. For the formula in (\ref{item3:J int}), note that
\begin{equation*}
\mathrm{ker}\,\Dolfp=H^{0,0}(\Sigma, f^*T^{1,0}M),\quad \mathrm{coker}\,\Dolfp=H^{0,1}(\Sigma, f^*T^{1,0}M).
\end{equation*}
The claim in (\ref{item4:J int}) follows with Serre duality.
\end{proof}
\begin{rem}
For a non-integrable almost complex structure $J$, the operators $\Dolfp$ and $L_f\bar{\partial}_J$ differ by an operator of order $0$, cf.~\cite[p.~28]{McS1}.
\end{rem}
\begin{rem}[cf.~\cite{McS1, McS2, TF}]\label{rem:discussion J integrable regular}
For an arbitrary smooth map $f\colon \Sigma\rightarrow M$, smooth sections of $f^*TM$ correspond to infinitesimal deformations of $f$. Suppose that $f$ is $J$-holomorphic. Then elements of
\begin{equation*}
\mathrm{Def}_J(f)=\mathrm{ker}\,L_f\bar{\partial}_J
\end{equation*}
correspond to infinitesimal deformations of $f$ through $J$-holomorphic curves. The vector space 
\begin{equation*}
\mathrm{Obs}_J(f)=\mathrm{coker}\,L_f\bar{\partial}_J
\end{equation*}
is called the obstruction space and the pair $(f,J)$ is called regular if $\mathrm{Obs}_J(f)=0$, i.e.~$L_f\bar{\partial}_J$ is surjective. If $(f,J)$ is regular, then $(f',J)$ is regular for all $J$-holomorphic curves $f'\colon\Sigma\rightarrow M$ in a small neighbourhood of $f$ (inside the space of all smooth maps $\Sigma\rightarrow M$). In this case, it follows that the local moduli space, i.e.~the set of all $J$-holomorphic curves $f'$ near $f$, is a smooth manifold of real dimension $2\mathrm{ind}_{\mathbb{C}}\Dirf$ with tangent space in $f$ given by $\mathrm{Def}_J(f)$.
\end{rem}
\begin{rem}
For a twisted $\mathrm{Spin}^c$-structure $S^c\otimes L$ with complex line bundle $L\rightarrow \Sigma$, as in Remark \ref{rem:twist Sc with L}, we can consider the spinor bundle $S^c\otimes L\otimes_{\mathbb{R}} f^*TM$. The choice of a Hermitian connection $B$ on $L$ then defines a connection $\nabla^{\scriptscriptstyle f\otimes B}$ on $S^c\otimes L\otimes_{\mathbb{R}} f^*TM$ with Dirac operator
\begin{equation*}
\Dirf_B\colon \Gamma(S^{c\pm}\otimes L\otimes_{\mathbb{R}} f^*TM)\longrightarrow \Gamma(S^{c\mp}\otimes L\otimes_{\mathbb{R}} f^*TM)
\end{equation*}
given by a generalization of Proposition \ref{prop:twisted Dirac Df formula}.
\end{rem}

\section{Relation to topological string theory}\label{sect:TQFT}

Dirac-harmonic maps on Riemann surfaces $\Sigma$ with the canonical $\mathrm{Spin}^c$-structure are related to topological string theory, introduced by Edward Witten \cite{W1, W2}. We combine the $\mathrm{Spin}^c$ spinor bundles 
\begin{align*}
S^c&=\underline{\mathbb{C}}\oplus K^{-1}\\
\bar{S}^c&=K\oplus \underline{\mathbb{C}}
\end{align*}
on the Riemann surface to a twisted complex spinor bundle
\begin{equation*}
\Delta=(S^c\oplus\bar{S}^c)\otimes f^*T^{\mathbb{C}}M
\end{equation*}
with Weyl spinor bundles
\begin{align*}
\Delta^+&=T_f^{1,0}M\oplus (K\otimes T_f^{1,0}M) \oplus T_f^{0,1}M\oplus (K\otimes T_f^{0,1}M)\\
\Delta^-&=(K^{-1}\otimes T_f^{1,0}M) \oplus T_f^{1,0}M \oplus (K^{-1}\otimes T_f^{0,1}M)\oplus T_f^{0,1}M.
\end{align*}
Here $(M,J,g,\omega)$ is a K\"ahler manifold of complex dimension $n$ and the pullback $f^*$ of $T^{1,0}M$ and $T^{0,1}M$ is abbreviated by an index $f$.
\begin{defn}
We define the following subbundles\footnote{We follow the conventions in \cite{W2}.}:
\begin{description}
\item[$\bm{+}$ {\bf twist}] 
\begin{align*}
\Delta^+_{(+)}&=T_f^{1,0}M\oplus (K\otimes T_f^{0,1}M)\\
\Delta^-_{(+)}&=T_f^{1,0}M\oplus (K^{-1}\otimes T_f^{0,1}M).
\end{align*}
\item[$\bm{-}$ {\bf twist}] 
\begin{align*}
\Delta^+_{(-)}&=(K\otimes T_f^{1,0}M)\oplus T_f^{0,1}M\\
\Delta^-_{(-)}&=(K^{-1}\otimes T_f^{1,0}M)\oplus T_f^{0,1}M.
\end{align*}
\end{description}
We also define the following spinor bundles:
\begin{description}
\item[{\bf A-model}] 
\begin{align*}
\Delta_A &= \Delta^+_{(+)}\oplus \Delta^-_{(-)}\\
\text{with sections} &\quad (\chi,\psi_z',\psi_{\bar{z}},\chi')
\end{align*}
\item[{\bf B-model}] 
\begin{align*}
\Delta_B &= \Delta^+_{(-)}\oplus \Delta^-_{(-)}\\
\text{with sections} &\quad (\rho_z,\tfrac{1}{2}(\eta'+\theta'),\rho_{\bar{z}},\tfrac{1}{2}(\eta'-\theta'))
\end{align*}
\end{description}
\end{defn}
To explain these definitions we consider the action functional \eqref{eqn:L[f,psi]}
\begin{equation*}
L[f,\psi]=\frac{1}{2}\int_\Sigma\left(|df|^2+\langle \psi, \Dirf\psi\rangle\right)\,\mathrm{dvol}_h.
\end{equation*}
The complete supersymmetric $\sigma$-model action functional also contains the quartic spinor term involving the Riemann curvature tensor of $g$, mentioned at the end of Section \ref{sect:background}. We ignore this term in the following discussion. 

We first consider the case where $(M,g)$ is a Riemannian manifold and the spinor a section $\psi\in \Gamma(S\otimes_{\mathbb{R}}f^*TM)$ for the spinor bundle $S$ of a spin structure on $\Sigma$. One allows a slightly more general situation where the Weyl spinor bundles come from different spin structures: Let $\Kh$ and $\Kbh$ be holomorphic square roots of $K$ and $\bar{K}$, not necessarily related by $\Kbh=\overline{\Kh}$. Then
\begin{equation*}
\psi_+\in\Gamma(\Kh\otimes_{\mathbb{R}} f^*TM),\quad \psi_-\in\Gamma(\Kbh\otimes_{\mathbb{R}} f^*TM).
\end{equation*}
The non-linear $\sigma$-model has $N=2$ supersymmetry generated by spinors
\begin{equation*}
\epsilon_-\in\Gamma(\Kmh),\quad \epsilon_+\in \Gamma(\Kbmh),
\end{equation*}
which are holomorphic and antiholomorphic sections of $\Kmh$ and $\Kbmh$, respectively.

Suppose that $(M,J,g,\omega)$ is a K\"ahler manifold of complex dimension $n$. We can decompose $T^{\mathbb{C}}M$ into the $(1,0)$- and $(0,1)$-part and denote the Weyl spinors by
\begin{align*}
(\psi_+,\psi_+')&\in (\Kh\otimes T_f^{1,0}M)\oplus (\Kh\otimes T_f^{0,1}M)\\
(\psi_-,\psi_-')&\in (\Kbh\otimes T_f^{1,0}M)\oplus (\Kbh\otimes T_f^{0,1}M).
\end{align*}
The non-linear $\sigma$-model now has $N=(2,2)$ supersymmetry generated by (anti)-holomorphic sections
\begin{equation}\label{eqn:alpha susy}
\alpha_-,\tilde{\alpha}_-\in\Gamma(\Kmh),\quad \alpha_+,\tilde{\alpha}_+\in \Gamma(\Kbmh).
\end{equation}
For a Riemann surface of genus $g_\Sigma\neq 1$ the canonical and anticanonical bundle are non-trivial, hence the sections in \eqref{eqn:alpha susy} have zeroes. In particular, the only covariantly constant sections, corresponding to global (rigid) supersymmetries, are identically zero. 

This can be remedied with the topological $+$ and $-$ twists, i.e.~using the $\mathrm{Spin}^c$-spinor bundle $S^c$ instead of the spinor bundle $S$. In the A-model the sections
\begin{equation*}
\alpha_-,\tilde{\alpha}_+\in\Gamma(\underline{\mathbb{C}})
\end{equation*}   
and in the B-model the sections
\begin{equation*}
\tilde{\alpha}_-,\tilde{\alpha}_+\in\Gamma(\underline{\mathbb{C}})
\end{equation*}
can be chosen covariantly constant. These sections yield a global fermionic symmetry $Q$ of the non-linear $\sigma$-model for arbitrary genus $g_\Sigma$, which implies that the A-model and B-model (for suitable target spaces) define topological quantum field theories (TQFTs).

We consider the A-model spinor bundle in more detail. The vector bundle $\Delta_A$ can be decomposed as
\begin{equation*}
\Delta_A = (S^c\otimes T^{1,0}_fM)\oplus (\bar{S}^c\otimes T^{0,1}_fM)
\end{equation*}
with sections
\begin{equation*}
(\Psi,\Psi'),\quad \Psi=(\chi,\psi_{\bar{z}}),\,\Psi'=(\psi_z',\chi').
\end{equation*}
The fermionic action \eqref{eqn:L[f,psi]} for the spinor bundle $\Delta_A$ can then be written as
\begin{equation*}
L_A[f,\Psi,\Psi']=\frac{1}{2}\int_\Sigma\left(|df|^2+\langle \Psi, \Dirfp\Psi\rangle+\langle \Psi', \Dirbfpp\Psi'\rangle\right)\,\mathrm{dvol}_h.
\end{equation*}
There is a complex antilinear bundle isomorphism
\begin{equation*}
S^c\otimes T^{1,0}_fM\stackrel{\cong}{\longrightarrow}\bar{S}^c\otimes T^{0,1}_fM
\end{equation*}
given by complex conjugation and exchanging positive and negative Weyl spinors, which induces a corresponding isomorphism between $\mathrm{ker}\,\Dirfp$ and $\mathrm{ker}\,\Dirbfpp$. Defining the numbers of zero modes
\begin{align*}
a&=\dim_{\mathbb{C}}\{(\chi,\chi')\mid \Dirfp\chi=0=\Dirbfpp\chi'\}\\
b&=\dim_{\mathbb{C}}\{(\psi_{\bar{z}},\psi_z')\mid \Dirfp\psi_{\bar{z}}=0=\Dirbfpp\psi_z'\},
\end{align*}
the index of the Dirac operator $\Dirfp$ is related to the so-called ghost number or $\mathrm{U}(1)_A$-anomaly by
\begin{equation*}
w=a-b=2\mathrm{ind}_{\mathbb{C}}\Dirfp=2n(1-g_\Sigma)+2c_1(A).
\end{equation*}

\section{Dirac-harmonic maps to K\"ahler manifolds}\label{sect:curvature term Rfpsi}
Let $(\Sigma,j,h)$ be a Riemann surface and $(M,J,g,\omega)$ a K\"ahler manifold of complex dimension $n$ with Levi--Civita connection $\nabla^M=\nabla^g$.

Let $f\colon \Sigma\rightarrow M$ be a smooth map and $\psi\in \Gamma(S^c\otimes_{\mathbb{R}} f^*TM)$ a twisted spinor. Then $(f,\psi)$ is called a Dirac-harmonic map if it is a critical point of the fermionic action functional \eqref{eqn:L[f,psi]} (with the spinor bundle $S$ replaced by $S^c$). The same proof as in \cite[Proposition 2.1]{CJLW2} for spin structures shows that a pair $(f,\psi)$ is a Dirac-harmonic map if and only if it satisfies the Euler--Lagrange equations \eqref{eqn:Dirac-harmonic map intro}.
\begin{defn}\label{defn:XA}
For $A\in H_2(M;\mathbb{Z})$ let
\begin{equation*}
\mathcal{X}_A=\mathrm{Map}(\Sigma,M;A)
\end{equation*}
be the set of all smooth maps $f\colon\Sigma\rightarrow M$ with $f_*[\Sigma]=A$, where $[\Sigma]\in H_2(\Sigma;\mathbb{Z})$ is the generator determined by the complex orientation of $\Sigma$.
\end{defn}
\begin{prop}\label{prop:J-hol tau(f) 0}
If $f\colon \Sigma\rightarrow M$ is $J$-holomorphic, then $f$ is harmonic and satisfies $\tau(f)=0$. More precisely, the absolute minima of the Dirichlet energy $L[f]$ on $\mathcal{X}_A$ are given by the $J$-holomorphic curves $f$ with $f_*[\Sigma]=A$. The Dirichlet energy of a $J$-holomorphic curve $f$ has value
\begin{equation*}
L[f]=\langle \omega,[A]\rangle,
\end{equation*}
where $\omega$ is the K\"ahler form on $M$. 
\end{prop}
\begin{proof} The vanishing of the tension field $\tau(f)$ for $J$-holomorphic curves $f$ is well-known, cf.~an example on \cite[p.~118]{ES}, and can be derived directly from formula \eqref{eqn:tension field tau f} with respect to a local orthonormal frame $\{e_1,e_2=je_1\}$, using that $df(e_2)=Jdf(e_1)$ and that the connection $\nabla^M=\nabla^g$ is torsion-free and Hermitian. The second part is proved in \cite[Lemma 2.2.1]{McS2} (note that deformations of $f$ do not change the integral homology class $f_*[\Sigma]$).
\end{proof}
\begin{rem}
More generally, if the target manifold is only almost K\"ahler, \cite[Lemma 2.2.1]{McS2} shows that $J$-holomorphic maps from closed Riemann surfaces are still absolute minima of the Dirichlet energy functional, hence harmonic maps. However, if $\nabla^M$ has torsion, the equation $\tau(f)=0$ does not necessarily follow. Dirac-harmonic maps for connections $\nabla^M$ with torsion have been studied in \cite{Br}. 
\end{rem}
The following statement appears in the proof of \cite[Theorem 1.1]{Sun} (more details on the definition of the curvature term $\mathcal{R}(f,\psi)$ can be found in Appendix \ref{sect:appendix B}).
\begin{prop}\label{prop:curv term zero} 
Let $f\colon\Sigma\rightarrow M$ be smooth map. Then 
\begin{equation*}
\mathcal{R}(f,\psi)=0
\end{equation*}
for all twisted spinors $\psi$ which are sections of one of the following subbundles of $S^c\otimes f^*T^{\mathbb{C}}M$ (using the notation of Section \ref{sect:TQFT}):
\begin{equation}\label{eqn:eigenbundles dvolh I}
\begin{split}
&S^{c+}\otimes (T_f^{1,0}M\oplus T_f^{0,1}M)\\
&S^{c-}\otimes (T_f^{1,0}M\oplus T_f^{0,1}M)\\
&(S^{c+}\otimes T_f^{1,0}M)\oplus (S^{c-}\otimes T_f^{0,1}M)\\
&(S^{c+}\otimes T_f^{0,1}M)\oplus (S^{c-}\otimes T_f^{1,0}M).
\end{split}
\end{equation}
\end{prop}
\begin{proof}
This can be proved as in \cite{Sun} by considering the expression (using the notation from Appendix \ref{sect:appendix B})
\begin{equation*}
2g(\mathcal{R}(f,\psi),f^*X)=\langle\psi,\Rf(X,\psi)\rangle.
\end{equation*}
Alternatively, consider a smooth map $f\colon\Sigma\rightarrow M$ with variation $f_t$ given by a vector field $X\in\Gamma(f^*TM)$. Any spinor $\psi\in\Gamma(S^c\otimes f^*T^{\mathbb{C}}M$) defines a spinor $\psi_t=\sum_\mu \psi_\mu\otimes f_t^*\partial_\mu$ with time-independent components $\psi_\mu$ with respect to local coordinates on $M$. By equation \eqref{eqn:var psi,Dpsi R(f,psi)}
\begin{equation*}
\left.\frac{d}{dt}\right|_{t=0}\int_\Sigma \langle \psi_t, \Dirft\psi_t\rangle\,\mathrm{dvol}_h=2\int_\Sigma g(\mathcal{R}(f,\psi),f^*X)\,\mathrm{dvol}_h.
\end{equation*}
For any variation $f_t$ the Dirac operator $\Dirft$ maps positive (negative) to negative (positive) Weyl spinors and preserves the $(1,0)$- and $(0,1)$-type of twisted spinors. Furthermore, the bundles $S^{c+}$ and $S^{c-}$ as well as $T^{1,0}M$ and $T^{0,1}M$ are orthogonal with respect to the Hermitian bundle metric. 

This implies for every section $\psi$ of the bundles in \eqref{eqn:eigenbundles dvolh I} that the corresponding spinor $\psi_t$ satisfies
\begin{equation*}
\langle \psi_t, \Dirft\psi_t\rangle = 0\quad\forall t.
\end{equation*}
\end{proof}
\begin{rem}
The first two bundles in \eqref{eqn:eigenbundles dvolh I} can be described as the $(\mp i)$-eigenspaces of the bundle automorphism $\mathrm{dvol}_h=\mathrm{dvol}_h\otimes \mathrm{Id}$ on $S^c\otimes f^*T^{\mathbb{C}}M$ with $\mathrm{dvol}_h^2=-\mathrm{Id}$ (cf.~equation \eqref{eqn:Clifford volume}). The other two bundles are the $(\pm 1)$-eigenspaces of the bundle automorphism $I=\mathrm{dvol}_h\otimes J$ on $S^c\otimes f^*T^{\mathbb{C}}M$ with $I^2=\mathrm{Id}$. 
\end{rem}
\begin{proof}[Proof of Theorem \ref{thm:J-hol curve Dirac-harmonic}] The first claim is a direct consequence of the Euler--Lagrange equations \eqref{eqn:Dirac-harmonic map intro} and Propositions \ref{prop:J-hol tau(f) 0}, \ref{prop:curv term zero} and \ref{prop:twisted Dirac Df formula}. The second claim follows because if all of the vector spaces are zero, then
\begin{equation*}
\mathrm{ind}_{\mathbb{C}}\Dirfp=\mathrm{ind}_{\mathbb{C}}\Dirfpp=0.
\end{equation*}
\end{proof}
\begin{rem}
A Dirac-harmonic map $(f,\psi)$ as in Theorem \ref{thm:J-hol curve Dirac-harmonic}, whose underlying map $f$ is harmonic, is called {\em uncoupled} in \cite{AG1}. The Dirac-harmonic maps $(f,\psi)$ in Theorem \ref{thm:J-hol curve Dirac-harmonic} have minimal bosonic action $L[f]$ in their homology class $A$.
\end{rem}
\begin{ex}\label{ex:sphere in CY}
Suppose that $(M,J,g,\omega)$ is a Calabi--Yau manifold of complex dimension $n$, hence $c_1(TM)=0$, and $f\colon \mathbb{C}\mathrm{P}^1\rightarrow M$ is a $J$-holomorphic sphere. If $(f,J)$ is regular, then the vector space $\mathrm{ker}\,\Dolfp$ has complex dimension $n$ and is the tangent space $\mathrm{Def}_J(f)$ in $f$ of the local moduli space of $J$-holomorphic spheres (compare with \cite[Remark 2.4]{G3}). For every $\psi\in\mathrm{ker}\,\Dolfp$, the pair $(f,\psi)$ is Dirac-harmonic.   
\end{ex}
\begin{defn}
Let $(\Sigma,j)$ be a fixed Riemann surface. For a class $A\in H_2(M;\mathbb{Z})$ we denote by $\mathcal{M}(A,J)$ the space of all $J$-holomorphic curves $f\colon \Sigma\rightarrow M$ with $f_*[\Sigma]=A$.
\end{defn}
\begin{proof}[Proof of Corollary \ref{cor:moduli M(A,J) Dirac-harmonic}] This follows, because under the assumptions $T_f\mathcal{M}(A,J)=\mathrm{ker}\,\Dolfp$ for all $f\in\mathcal{M}(A,J)$ (cf.~Remark \ref{rem:discussion J integrable regular}).
\end{proof}
\begin{ex}\label{ex:sphere in Kaehler surface} Suppose that $(M,J,g,\omega)$ is a K\"ahler surface and $f\colon \mathbb{C}\mathrm{P}^1\rightarrow M$ an embedded $J$-holomorphic sphere representing a class $A$ of self-intersection $A^2=A\cdot A\geq -1$. Then every $f'\in \mathcal{M}(A,J)$ is an embedding and $(f',J)$ is regular (see \cite[Corollary 3.5.4]{McS1}). By the adjunction formula
\begin{equation*}
-2=A^2-c_1(A),
\end{equation*}
hence $\mathcal{M}(A,J)$ is a smooth manifold of real dimension $8+2A^2\geq 6$. The tangent bundle $T\mathcal{M}(A,J)$ is a complex vector bundle and consists of Dirac-harmonic maps.
\end{ex}

\section{Generalization to twisted $\mathrm{Spin}^c$-structures on $\Sigma$}
We consider the following generalization for the same setup as in Section \ref{sect:curvature term Rfpsi}: Let $L\rightarrow \Sigma$ be a holomorphic Hermitian line bundle with Chern connection $\nabla$ and Dolbeault operator
\begin{equation*}
\bar{\partial}=\bar{\partial}_{\nabla}\colon \Gamma(L)\longrightarrow \Omega^{0,1}(L).
\end{equation*}
Then $\mathfrak{s}^c\otimes L$ is a $\mathrm{Spin}^c$-structure with holomorphic spinor bundles
\begin{align*}
S^{c+}\otimes L&=L\\
S^{c-}\otimes L&=K^{-1}\otimes L
\end{align*}
and Dolbeault--Dirac operator
\begin{equation*}
D=\sqrt{2}(\bar{\partial}+\bar{\partial}^*)\colon \Gamma(S^{c\pm}\otimes L)\longrightarrow\Gamma(S^{c\mp}\otimes L).
\end{equation*} 
\begin{lem}
Let $f\colon\Sigma\rightarrow M$ be a smooth map. The twisted Dirac operator
\begin{equation*}
\Dirf\colon \Gamma(S^{c\pm}\otimes L\otimes f^*T^{\mathbb{C}}M)\longrightarrow \Gamma(S^{c\mp}\otimes L\otimes f^*T^{\mathbb{C}}M)
\end{equation*}
decomposes into the sum $\Dirf=\Dirfp+\Dirfpp$ of two twisted Dolbeault--Dirac operators
\begin{align*}
\Dirfp&=\sqrt{2}(\Dolfp+\Dolfps)\\
\Dirfpp&=\sqrt{2}(\Dolfpp+\Dolfpps).
\end{align*}
\end{lem}
In this situation we can define Dirac-harmonic maps $(f,\psi)$ as solutions of the analogue of the system of equations \eqref{eqn:Dirac-harmonic map intro}. 
\begin{cor}\label{cor:Dirac-harmonic twisted hol L}
Let $f\colon\Sigma\rightarrow M$ be a $J$-holomorphic curve with $A=f_*[\Sigma]$. If $\psi\in \Gamma(S^c\otimes L \otimes f^*T^{\mathbb{C}}M)$ is an element of one of the following vector spaces, then $(f,\psi)$ is Dirac-harmonic:
\begin{equation}\label{eqn:J-hol curve Dirac-harmonic vector spaces II}
\begin{array}{ll} \mathrm{ker}\,\Dolfp\oplus \mathrm{ker}\,\Dolfpp, & \mathrm{ker}\,\Dolfps\oplus \mathrm{ker}\,\Dolfpps \\ \mathrm{ker}\,\Dolfp\oplus \mathrm{ker}\,\Dolfpps, & \mathrm{ker}\,\Dolfpp\oplus \mathrm{ker}\,\Dolfps.  \end{array}
\end{equation}
By the Hirzebruch--Riemann--Roch Theorem
\begin{align*}
\mathrm{ind}_{\mathbb{C}}\Dirfp&=n(1-g_\Sigma+c_1(L))+c_1(A)\\
\mathrm{ind}_{\mathbb{C}}\Dirfpp&=n(1-g_\Sigma+c_1(L))-c_1(A),
\end{align*}
where we write $c_1(L)$ for $\langle c_1(L),[\Sigma]\rangle$.
\end{cor}
\begin{rem}\label{rem:hol supercurve}
A Dirac-harmonic map $(f,\psi)$, where $f$ is a $J$-holomorphic curve and $\psi\in\mathrm{ker}\,\Dolfp$, is a $(\nabla^g,J)$-holomorphic supercurve as studied in \cite{G3}, cf.~also \cite{G1}.
\end{rem}
\begin{ex} Consider again the situation in Example \ref{ex:sphere in Kaehler surface} of a K\"ahler surface $(M,J,g,\omega)$ with an embedded $J$-holomorphic sphere $f\colon \mathbb{C}\mathrm{P}^1\rightarrow M$ of self-intersection $A^2=A\cdot A\geq -1$ and smooth moduli space $\mathcal{M}(A,J)$. Let $L\rightarrow\Sigma$ be a holomorphic line bundle with $c_1(L)>0$. Then
\begin{equation*}
c_1(L\otimes f^*T^{1,0}M)=2c_1(L)+c_1(A)\geq 3
\end{equation*}
and the arguments in \cite[Section 3.5]{McS1} using the Kodaira vanishing theorem show that $\mathrm{coker}\,\Dolfp=0$. Hence the complex vector space $\mathrm{ker}\,\Dolfp$ has constant dimension
\begin{equation*}
\dim_{\mathbb{C}}\mathrm{ker}\,\Dolfp=4+A^2+2c_1(L)
\end{equation*}
for all $f\in\mathcal{M}(A,J)$. There is a complex vector bundle over the infinite-dimensional manifold $\mathcal{X}_A$ from Definition \ref{defn:XA} with fibre $\Gamma(L\otimes f^*T^{1,0}M)$ over $f\in \mathcal{X}_A$. Since $\mathcal{M}(A,J)$ is a submanifold of $\mathcal{X}_A$, it follows that the subset of Dirac-harmonic maps $(f,\psi)$ with
\begin{equation*}
f\in \mathcal{M}(A,J),\quad \psi\in\mathrm{ker}\,\Dolfp\subset\Gamma(L\otimes f^*T^{1,0}M) 
\end{equation*}
is a smooth complex vector bundle $E$ over $\mathcal{M}(A,J)$ of rank
\begin{equation*}
\mathrm{rk}_{\mathbb{C}}E=4+A^2+2c_1(L).
\end{equation*}
In particular, for $L=K^{\otimes (-q)}$ with integers $q\geq 1$, we have $c_1(L)=2q$ and the complex vector bundle $E$ over $\mathcal{M}(A,J)$ of Dirac-harmonic maps has rank 
\begin{equation*}
\mathrm{rk}_{\mathbb{C}}E=4+A^2+4q,
\end{equation*}
which becomes arbitrarily large for $q\gg 1$.
\end{ex}
\begin{ex}\label{ex:Dirac harmonic spin structure}
Let $\mathfrak{s}$ be a spin structure on $\Sigma$ and $L=K^{\scriptscriptstyle\frac{1}{2}}$ the associated holomorphic square root of the canonical bundle $K$. Then $S\cong S^c\otimes K^{\scriptscriptstyle\frac{1}{2}}$ is the spinor bundle of $\mathfrak{s}$ with spin Dirac operator $D=\sqrt{2}(\bar{\partial}+\bar{\partial}^*)$ (cf.~\cite{H}) and
\begin{align*}
\mathrm{ind}_{\mathbb{C}}\Dirfp&=c_1(A)\\
\mathrm{ind}_{\mathbb{C}}\Dirfpp&=-c_1(A).
\end{align*}
The vector spaces in \eqref{eqn:J-hol curve Dirac-harmonic vector spaces II} are called $V^\pm_{even}$ and $V^\pm_{odd}$ in the proof of \cite[Theorem 1.1]{Sun}. 
\end{ex}

\subsection*{Acknowledgements} I would like to thank Uwe Semmelmann for discussions on a previous version of this paper and the referee for helpful remarks.

\appendix
\renewcommand{\thesection}{\Alph{section}}

\section{}\label{sect:appendix A}
Let $(\Sigma,j,h)$ be a closed Riemann surface with complex structure $j$ and compatible Riemannian metric $h$. We fix some notation for the decomposition of tangent vectors and $1$-forms into those of type $(1,0)$ and $(0,1)$.

The almost complex structure $j$ on $T\Sigma$ extends canonically to a complex linear isomorphism on $T^{\mathbb{C}}\Sigma=T\Sigma\otimes_{\mathbb{R}} \mathbb{C}$ and we decompose
\begin{equation}\label{eqn:TSigma 10 01}
T^{\mathbb{C}}\Sigma=T^{1,0}\oplus T^{0,1}
\end{equation}
into the complex $(+i)$- and $(-i)$-eigenspaces of $j$. The Riemannian metric $h$ extends to a Hermitian bundle metric on $T^{\mathbb{C}}\Sigma$ and the decomposition in \eqref{eqn:TSigma 10 01} is orthogonal.

The dual space $(T^{\mathbb{C}}\Sigma)^*$ of complex linear $1$-forms on $T^{\mathbb{C}}\Sigma$ decomposes into
\begin{equation*}
(T^{\mathbb{C}}\Sigma)^*=\Lambda^{1,0}\oplus \Lambda^{0,1},
\end{equation*}
where $\Lambda^{1,0}=K$ and $\Lambda^{0,1}=K^{-1}$ are the bundles of complex linear $1$-forms on $T^{1,0}$ and $T^{0,1}$. We have
\begin{align*}
\alpha\circ j&=i\alpha\quad\forall \alpha\in\Lambda^{1,0}\\
\beta\circ j&=-i\beta\quad\forall \beta\in\Lambda^{0,1}.
\end{align*}
If $\tau\in (T^{\mathbb{C}}\Sigma)^*$ is a $1$-form, then its decomposition into $(1,0)$- and $(0,1)$-components is given by
\begin{equation*}
\tau=\tau^{1,0}+\tau^{0,1}
\end{equation*}
with
\begin{equation*}
\tau^{1,0}=\tfrac{1}{2}(\tau - i\tau\circ j),\quad \tau^{0,1}=\tfrac{1}{2}(\tau + i\tau\circ j).
\end{equation*}
Let $(e_1,e_2)$ with $e_2=je_1$ be a local $h$-orthonormal basis of $T\Sigma$. Then
\begin{equation*}
\epsilon= \tfrac{1}{\sqrt{2}}(e_1-ie_2),\quad \bar{\epsilon}=\tfrac{1}{\sqrt{2}}(e_1+ie_2)
\end{equation*}
are local unit basis vectors of $T^{1,0}$ and $T^{0,1}$. We extend the dual real basis $(e_1^*,e_2^*)$ of $T^*\Sigma$ to a basis of complex linear $1$-forms of $(T^{\mathbb{C}}\Sigma)^*$. Then 
\begin{equation*}
\kappa = \tfrac{1}{\sqrt{2}}(e_1^* + ie_2^*),\quad \bar{\kappa}= \tfrac{1}{\sqrt{2}}(e_1^* - ie_2^*)
\end{equation*}
are the dual local unit basis vectors of $K$ and $K^{-1}$.

\section{}\label{sect:appendix B}
We summarize the definition of the curvature term $\mathcal{R}(f,\psi)$ that appears in the Euler--Lagrange equations \eqref{eqn:Dirac-harmonic map intro} for Dirac-harmonic maps. Let $(\Sigma,j,h)$ be a Riemann surface, $(M^n,g)$ a Riemannian manifold and $f\colon \Sigma\rightarrow M$ a smooth map. We denote by
\begin{equation*}
R\colon TM\times TM\times TM\longrightarrow TM
\end{equation*}
the curvature tensor, where we use the sign convention
\begin{equation*}
R(X,Y)Z=[\nabla^g_X,\nabla^g_Y]Z-\nabla^g_{[X,Y]}Z.
\end{equation*}
There is an induced map
\begin{align*}
TM\times (S^c\otimes f^*T^{\mathbb{C}}M)&\longrightarrow T^*\Sigma \times (S^c\otimes f^*T^{\mathbb{C}}M)\\
(X,\phi\otimes f^*Z)&\longmapsto \phi\otimes f^*(R(X,df(\cdot))Z).
\end{align*}
Composing with Clifford multiplication
\begin{equation*}
\gamma\colon T^*\Sigma\times S^c\longrightarrow S^c
\end{equation*}
we get the map
\begin{align*}
\Rf\colon TM\times (S^c\otimes f^*T^{\mathbb{C}}M)&\longrightarrow S^c\otimes f^*T^{\mathbb{C}}M\\
(X,\psi)&\longmapsto \Rf(X,\psi).
\end{align*}
\begin{defn}
We define
\begin{equation*}
\mathcal{R}(f,\cdot)\colon S^c\otimes f^*T^{\mathbb{C}}M\longrightarrow f^*TM,\quad \psi\longmapsto \mathcal{R}(f,\psi)
\end{equation*}
by
\begin{equation*}
g(\mathcal{R}(f,\psi),f^*X)=\frac{1}{2}\langle\psi,\Rf(X,\psi)\rangle\quad\forall f^*X\in f^*TM.
\end{equation*}
\end{defn}
With respect to a local orthonormal frame $e_1,e_2$ for $T\Sigma$ we can write
\begin{equation*}
\Rf(X,\phi\otimes f^*Z)=\sum_{\alpha=1}^2 e_\alpha\cdot\phi\otimes f^*(R(X,df(e_\alpha))Z).
\end{equation*}
With the components of the curvature tensor $R$ with respect to a local frame $\{y_k\}_{k=1}^n$
\begin{equation*}
\sum_{i=1}^{n}R_{ijml}y_i=R(y_m,y_l)y_j
\end{equation*}
we obtain the original formula for the definition of the curvature term $\mathcal{R}$ in \cite{CJLW2}:
\begin{equation*}
\mathcal{R}(f,\psi)=\frac{1}{2}\sum_{i,j,m,l,\alpha}R_{ijml}df(e_\alpha)_l\langle \psi_i,e_\alpha\cdot \psi_j\rangle f^*y_m.
\end{equation*}
The symmetries
\begin{equation*}
R_{ijml}=-R_{jiml},\quad \overline{\langle \psi_i,e_\alpha\cdot \psi_j\rangle}=-\langle \psi_j,e_\alpha\cdot \psi_i\rangle
\end{equation*}
imply that $\mathcal{R}(f,\psi)$ is indeed a real vector in $f^*TM$. 

Suppose that $f_t$ a variation of the smooth map $f\colon\Sigma\rightarrow M$ with
\begin{equation*}
f_0=f,\quad \left.\frac{df_t}{dt}\right|_{t=0}=f^*X\in \Gamma(f^*TM).
\end{equation*}
Let $\phi,\phi'\in \Gamma(S^c)$ be time-independent spinors on $\Sigma$, $Z,Z'$ time-independent vector fields on $M$ and define spinors
\begin{equation*}
\psi_t=\phi\otimes f_t^*Z,\quad \psi'_t=\phi'\otimes f_t^*Z'\in\Gamma(S^c\otimes_{\mathbb{R}}f^*TM). 
\end{equation*}
\begin{defn}\label{defn:nablaXg psi}
We set $df_-(e_\alpha)$ for the vector field $df_t(e_\alpha)$ along $f_t$ and
\begin{align*}
\nabla^g_X\psi&=\phi\otimes f^*\nabla^g_XZ\\
\nabla^g_X\nabla^f_{e_\alpha}\psi&=\nabla^h_{e_\alpha}\phi\otimes f^*\nabla^g_XZ+\phi\otimes f^*\nabla^g_X\nabla^g_{df_-(e_\alpha)}Z\\
&=\nabla^h_{e_\alpha}\phi\otimes f^*\nabla^g_XZ+\phi\otimes f^*(\nabla^g_{df(e_\alpha)}\nabla^g_XZ+R(X,df(e_\alpha))Z).
\end{align*}
In the last line we used that $[X,df_-(e_\alpha)]=0$, since $f_t$ is generated (to first order) by the flow of $X$.
\end{defn}
We calculate (cf.~the proof of \cite[Proposition 2.1]{CJLW2})
\begin{align*}
&\left.\frac{d}{dt}\right|_{t=0}\langle \psi'_t, \Dirft\psi_t\rangle\\
&=\left.\frac{d}{dt}\right|_{t=0}\sum_{\alpha=1}^2\langle \phi'\otimes f_t^*Z', e_\alpha\cdot ((\nabla^h_{e_\alpha}\phi)\otimes f_t^*Z+\phi\otimes f_t^*\nabla^g_{df_t(e_\alpha)}Z)\rangle\\ 
&=\sum_{\alpha=1}^2\left(\langle \phi',e_\alpha\cdot(\nabla^h_{e_\alpha}\phi)\rangle L_Xg(Z',Z)+ \langle\phi',e_\alpha\cdot\phi\rangle L_Xg(Z',\nabla^g_{df_-(e_\alpha)}Z)\right)\\
&=\langle \nabla^g_X\psi',\Dirf\psi\rangle+\langle \psi',\Dirf\nabla^g_X\psi\rangle+\langle \psi',\Rf(X,\psi)\rangle.
\end{align*}
In particular, for $\psi'=\psi$ and $\Dirf\psi=0$ we get formula \eqref{eqn:var psi,Dpsi R(f,psi)}, using that $\Dirf$ is formally self-adjoint.

\bibliographystyle{amsplain}

\begin{thebibliography}{999}

\bibitem{AF} L.~Alvarez-Gaum\'e, D.~Z.~Freedman, {\em The background field method and the ultraviolet structure of the supersymmetric nonlinear $\sigma$-model}, Annals Phys.~{\bf 134} (1981) 85--109.

\bibitem{AG1} B.~Ammann, N.~Ginoux, {\em Dirac-harmonic maps from index theory}, Calc.~Var.~Partial Differential Equations {\bf 47} (2013), no.~3-4, 739--762.

\bibitem{AG2} B.~Ammann, N.~Ginoux, {\em Some examples of Dirac-harmonic maps}, Lett.~Math.~Phys.~{\bf 109} (2019), no.~5, 1205--1218.

\bibitem{At} M.~F.~Atiyah, {\em Riemann surfaces and spin structures}, Ann.~Sci.~\'Ecole Norm.~Sup.~(4) {\bf 4} (1971) 47--62.

\bibitem{BGV} N.~Berline, E.~Getzler, M.~Vergne, {\sl Heat kernels and Dirac operators}. Corrected reprint of the 1992 original. Grundlehren Text Editions. Springer-Verlag, Berlin, 2004. x+363 pp.

\bibitem{Br} V.~Branding, {\em Dirac-harmonic maps with torsion}, Commun.~Contemp.~Math.~{\bf 18} (2016), no.~4, 1550064, 19 pp.

\bibitem{CMPF} C.~G.~Callan, Jr., D.~Friedan, E.~J.~Martinec, M.~J.~Perry, {\em Strings in background fields}, Nucl.~Phys.~{\bf B262} (1985) 593--609.

\bibitem{CJLW1} Q.~Chen, J.~Jost, J.~Li, G.~Wang, {\em Regularity theorems and energy identities for Dirac-harmonic maps}, Math.~Z.~{\bf 251} (2005), no.~1, 61--84.

\bibitem{CJLW2} Q.~Chen, J.~Jost, J.~Li, G.~Wang, {\em Dirac-harmonic maps}, Math.~Z.~{\bf 254} (2006), no.~2, 409--432.

\bibitem{CJSZ} Q.~Chen, J.~Jost, L.~Sun, M.~Zhu, {\em Dirac-harmonic maps between Riemann surfaces}, Asian J.~Math.~{\bf 23} (2019), no.~1, 107--125.

\bibitem{CJW} Q.~Chen, J.~Jost, G.~Wang, {\em Liouville theorems for Dirac-harmonic maps}, J.~Math.~Phys.~{\bf 48} (2007), no.~11, 113517, 13 pp.

\bibitem{ES} J.~Eells, Jr., J.~H.~Sampson, {\em Harmonic mappings of Riemannian manifolds}, Amer.~J.~Math.~{\bf 86} (1964) 109--160.

\bibitem{Fr} T.~Friedrich, {\sl Dirac operators in Riemannian geometry}. Graduate Studies in Mathematics, 25. American Mathematical Society, Providence, RI, 2000. xvi+195

\bibitem{Gin} N.~Ginoux, {\sl The Dirac spectrum}. Lecture Notes in Mathematics, 1976. Springer-Verlag, Berlin, 2009. xvi+156 pp.

\bibitem{G1} J.~Groeger, {\em Holomorphic supercurves and supersymmetric sigma models}, J.~Math.~Phys.~{\bf 52} (2011), no.~12, 123505, 21 pp.

\bibitem{G2} J.~Groeger, {\em Compactness for holomorphic supercurves}, Geom.~Dedicata {\bf 170} (2014), 347--384.

\bibitem{G3} J.~Groeger, {\em Transversality for holomorphic supercurves}, Math.~ Phys.~Anal.~Geom.~{\bf 17} (2014), no.~3-4, 341--367. 

\bibitem{H} N.~Hitchin, {\em Harmonic spinors}, Advances in Math.~{\bf 14} (1974), 1--55.

\bibitem{JKTWZ} J.~Jost, E.~Ke\ss ler, J.~Tolksdorf, R.~Wu, M.~Zhu, {\em Regularity of solutions of the nonlinear sigma model with gravitino}, Comm.~Math.~Phys.~{\bf 358} (2018), no.~1, 171--197.

\bibitem{JMZ} J.~Jost, X.~Mo, M.~Zhu, {\em Some explicit constructions of Dirac-harmonic maps}, J.~Geom.~Phys.~{\bf 59} (2009), no.~11, 1512--1527.

\bibitem{LM} H.~B.~Lawson, Jr., M.-L.~Michelsohn, {\em Spin geometry}, Princeton Mathematical Series, {\bf 38}. Princeton University Press, Princeton, NJ, 1989. xii+427 pp. ISBN: 0-691-08542-0

\bibitem{McS1} D.~McDuff, D.~Salamon, {\sl J-holomorphic curves and quantum cohomology}. University Lecture Series {\bf 6}. American Mathematical Society, Providence, RI, 1994. viii+207 pp.

\bibitem{McS2} D.~McDuff, D.~Salamon, {\sl J-holomorphic curves and symplectic topology}. Second edition. American Mathematical Society Colloquium Publications, {\bf 52}. American Mathematical Society, Providence, RI, 2012. xiv+726 pp.

\bibitem{Mo} X.~Mo, {\em Some rigidity results for Dirac-harmonic maps}, Publ.~Math. ~Debrecen {\bf 77} (2010), no.~3-4, 427--442.

\bibitem{Sun} L.~Sun, {\em A note on the uncoupled Dirac-harmonic maps from K\"ahler spin manifolds to K\"ahler manifolds}, Manuscripta Math.~{\bf 155} (2018), no.~1-2, 197--208.

\bibitem{TF} M.~F.~Tehrani, K.~Fukaya, {\em Gromov-Witten theory via Kuranishi structures}. in: Virtual fundamental cycles in symplectic topology, 111--252, Math.~ Surveys Monogr.~{\bf 237}, Amer.~Math.~Soc., Providence, RI, 2019.

\bibitem{W1} E.~Witten, {\em Topological sigma models}, Comm.~Math.~Phys.~{\bf 118} (1988), no.~3, 411--449.

\bibitem{W2} E.~Witten, {\em Mirror manifolds and topological field theory}. Essays on mirror manifolds, 120--158, Int.~Press, Hong Kong, 1992.

\bibitem{Yang} L.~Yang, {\em A structure theorem of Dirac-harmonic maps between spheres}, Calc.~Var.~Partial Differential Equations {\bf 35} (2009), no. 4, 409--420.


\end{thebibliography}

\bigskip
\bigskip

\end{document}